\documentclass{amsart}
\usepackage{amsmath,amssymb,amsfonts}
\usepackage{color}
\usepackage[colorlinks=true,linkcolor=blue,urlcolor=blue,citecolor   = red]{hyperref}

\newtheorem{theorem}{Theorem}[section]
\newtheorem{proposition}[theorem]{Proposition}
\newtheorem{definition}[theorem]{Definition}
\newtheorem{lemma}[theorem]{Lemma}
\newtheorem{corollary}[theorem]{Corollary}
\newtheorem*{remark}{Remark}

\begin{document}
\title[Some characterizations of almost limited sets and applications]{Some
characterizations of almost limited sets\\
and applications}
\author{Nabil Machrafi}
\address[N. Machrafi]{Universit\'{e} Ibn Tofail, Facult\'{e} des Sciences, D%
\'{e}partement de Math\'{e}\-matiques, B.P. 133, K\'{e}nitra, Morocco.}
\email{nmachrafi@gmail.com}
\author{Aziz Elbour}
\address[A. Elbour]{Universit\'{e} Moulay Isma\"{\i}l, Facult\'{e} des
Sciences et Techniques, D\'{e}partement de Math\'{e}\-matiques, B.P. 509,
Errachidia, Morocco.}
\email{azizelbour@hotmail.com}
\author{Mohammed Moussa}
\address[M. Moussa]{Universit\'{e} Ibn Tofail, Facult\'{e} des Sciences, D%
\'{e}partement de Math\'{e}\-matiques, B.P. 133, K\'{e}nitra, Morocco.}
\email{mohammed.moussa09@gmail.com}

\begin{abstract}
Recently, J.X. Chen et al. introduced and studied the class of almost
limited sets in Banach lattices. In this paper we establish some
characterizations of almost limited sets in Banach lattices (resp. wDP*
property of Banach lattices), that generalize some results obtained by J.X.
Chen et al.. Also, we introduce and study the class of the almost limited
operators, which maps the closed unit bull of a Banach space to an almost
limited subset of a Banach lattice. Some results about the relationship
between the class of almost limited operators and that of limited (resp. L-
and M-weakly compact, resp. compact) operators are presented.
\end{abstract}

\maketitle

\section{Introduction}

Throughout this paper $X,$ $Y$ will denote real Banach spaces, and $E,\,F$
will denote real Banach lattices. $B_{X}$ is the closed unit ball of $X$.

Let us recall that a norm bounded subset $A$ of a Banach space $X$ is called 
\emph{limited} \cite{BD}, if every weak$^{\ast }$ null sequence $\left(
f_{n}\right) $ in $X^{\ast }$ converges uniformly to zero on $A$, that is, $%
\sup_{x\in A}\left\vert f_{n}\left( x\right) \right\vert \rightarrow 0$. We
know that every relatively compact subset of $X$ is limited. Recently, J.X.
Chen et al. \cite{chen} introduced and studied the class of almost limited
sets in Banach lattices. A norm bounded subset $A$ of a Banach lattice $E$
is said to be \emph{almost limited}, if every disjoint weak$^{\ast }$ null
sequence $(f_{n})$ in $E^{\ast }$ converges uniformly to zero on $A$.

The aim of this paper is to establish some characterizations of almost
limited sets (resp. wDP$^{\ast }$\ property). As consequence, we give a
generalization of Theorems 2.5 and 3.2 of \cite{chen} (see Corollary \ref%
{corT(A)almostlimited} and Theorem \ref{bisequence}). After that, we
introduce and study the \emph{dual Schur property} in Banach lattices (Sect.
3). Finally, using the almost limited sets, we define a new class of
operators so-called \emph{almost limited} (definition \ref{def1}), and we
establish some relationship between the class of almost limited operators
and that of limited (resp. L- and M-weakly compact, resp. compact) operators
(Sect. 4). Our terminology and notations are standard--we follow by \cite%
{AB3, WN4}.

\section{Almost limited sets in Banach lattices}

Recall that the lattice operations of $E^{\ast }$\ are said to be weak$%
^{\ast }$ sequentially continuous whenever $f_{n}\overset{w^{\ast }}{%
\rightarrow }0$ in $E^{\ast }$ implies $\left\vert f_{n}\right\vert \overset{%
w^{\ast }}{\rightarrow }0$ in $E^{\ast }$.

An order interval of a Banach Lattice $E$ is not necessary almost limited
(resp. limited). In fact, the order interval $\left[ -\mathbf{1,1}\right] $\
of the Banach lattice $c$ is not almost limited (and hence not limited),
where $\mathbf{1}:=\left( 1,1,...\right) \in c$ \cite[Remark 2.4(1)]{chen}.

The next proposition characterizes the Banach lattices on which every order
interval is almost limited (resp. limited).

\begin{proposition}
\label{order interval}For a Banach lattice $E$ the following statements hold:

\begin{enumerate}
\item Every order interval of $E$ is almost limited if and only if $%
\left\vert f_{n}\right\vert \overset{w^{\ast }}{\rightarrow }0$ for each
disjoint weak$^{\ast }$ null sequence $\left( f_{n}\right) $ in $E^{\ast }$.

\item Every order interval of $E$ is limited if and only if the lattice
operations of $E^{\ast }$ are weak$^{\ast }$ sequentially continuous.
\end{enumerate}
\end{proposition}

\begin{proof}
Is a simple consequence of the Theorem 3.55 of \cite{AB3}.
\end{proof}

It should be noted, by Proposition 1.4 of \cite{WN2012} or Lemma 2.2 of \cite%
{chen}, that if $E$ is a $\sigma $-Dedekind complete Banach lattice then it
satisfies the following property: 
\begin{equation}
f_{m}\perp f_{k}\text{ in }E^{\ast }\text{ and }f_{n}\overset{w^{\ast }}{%
\rightarrow }0\text{ implies }\left\vert f_{n}\right\vert \overset{w^{\ast }}%
{\rightarrow }0.  \tag{d}
\end{equation}%
So, in a $\sigma $-Dedekind complete Banach lattice every order interval is
almost limited.

It is worth to note that the property $(\text{d})$ does not characterize the 
$\sigma $-Dedekind complete Banach lattices. In fact, the Banach lattice $%
l^{\infty }/c_{0}$\ has the property $(\text{d})$ but it is not $\sigma $%
-Dedekind complete \cite[Remark 1.5]{WN2012}.

Also, clearly if the lattice operations of $E^{\ast }$ are weak$^{\ast }$
sequentially continuous then $E$ has the property $(\text{d})$. The converse
is false in general. In fact, the Dedekind complete Banach lattice $%
l^{\infty }$ has the property $(\text{d})$ but the lattice operations of $%
\left( l^{\infty }\right) ^{\ast }$ are not weak$^{\ast }$ sequentially
continuous.

Recall that a subset $A$ of a Banach lattice $E$ is said to be \emph{almost
order bounded} if for every $\epsilon >0$ there exists some $u\in E^{+}$
such that $A\subseteq \left[ -u,u\right] +\epsilon B_{E}$, equivalently, if
for every $\epsilon >0$ there exists some $u\in E^{+}$ such that $\| \left(
\left\vert x\right\vert -u\right) ^{+}\| \leq \epsilon $ for all $x\in A$.

To prove that the order intervals can be replaced by the almost order
bounded subsets of $E$ in the Proposition \ref{order interval}, we need the
following Lemma.

\begin{lemma}
\label{lemme1}Let $A$\ be a norm bounded subset of a Banach lattice $E$. If
for every $\varepsilon >0$\ there exists some limited (resp. almost limited)
subset $A_{\varepsilon }$\ of $E$\ such that $A\subseteq A_{\varepsilon
}+\varepsilon B_{E}$, then $A$\ is limited (resp. almost limited).
\end{lemma}

\begin{proof}
Let $(f_{n})$ be a weak$^{\ast }$ null (resp. disjoint weak$^{\ast }$ null)
sequence in $E^{\ast }$, let $\varepsilon >0$. Pick some $M>0$ with $%
\left\Vert f_{n}\right\Vert \leq M$ for all $n$. By hypothesis, there exists
some limited (resp. almost limited) subset $A_{\varepsilon }$ of $E$ such
that $A\subseteq A_{\varepsilon }+\frac{\varepsilon }{2M}B_{E}$, and hence $%
\sup_{x\in A}\left\vert f_{n}(x)\right\vert \leq \sup_{x\in A_{\varepsilon
}}\left\vert f_{n}(x)\right\vert +\frac{\varepsilon }{2}$.

On the other hand, as $A_{\varepsilon }$ is limited (resp. almost limited),
there exists some $n_{0}$ with $\sup_{x\in A_{\varepsilon }}\left\vert
f_{n}(x)\right\vert <\frac{\varepsilon }{2}$ for all $n\geq n_{0}$.

Thus, $\sup_{x\in A}\left\vert f_{n}(x)\right\vert \leq \varepsilon $\ for
all $n\geq n_{0}$. This implies $\sup_{x\in A}\left\vert f_{n}(x)\right\vert
\rightarrow 0$, and then $A$\ is limited (resp. almost limited).
\end{proof}

Now, the following result is a simple consequence of Proposition \ref{order
interval} and Lemma \ref{lemme1}.

\begin{proposition}
\label{almostorderbounded}For a Banach lattice $E$ the following statements
hold:

\begin{enumerate}
\item The lattice operations of $E^{\ast }$ are weak$^{\ast }$ sequentially
continuous if and only if every almost order bounded subset of $E$ is
limited.

\item $E$ has the property $(d)$ if and only if every almost order bounded
subset of $E$ is almost limited.
\end{enumerate}
\end{proposition}

\begin{corollary}
\label{cor1}If $E$ is a $\sigma $-Dedekind complete Banach lattice then
every almost order bounded set in $E$ is almost limited.
\end{corollary}

To state our next result, we need the following lemma which is just a
particular case of Theorem 2.4 of \cite{DF}.

\begin{lemma}
\label{DF}Let $E$ be a Banach lattice, and let $\left( f_{n}\right) \subset
E^{\ast }$ be a sequence with $\left\vert f_{n}\right\vert \overset{w^{\ast }%
}{\rightarrow }0$. If $A\subset E$ is a norm bounded and solid set such that 
$f_{n}(x_{n})\rightarrow 0$ for every disjoint sequence $\left( x_{n}\right)
\subset A^{+}=A\cap E^{+}$, then $\sup_{x\in A}\left\vert f_{n}\right\vert
\left( x\right) \rightarrow 0$.
\end{lemma}

It is obvious that all relatively compact sets and all limited sets in a
Banach lattice are almost limited. The converse does not hold in general.
For example, the closed unit ball $B_{\ell ^{\,\infty }}$ is an almost
limited set in $\ell ^{\infty }$, but is not either compact or limited.

\begin{theorem}
\label{solidalmostlimited}For a Banach lattice $E$ the following statements
hold:

\begin{enumerate}
\item If the lattice operations of $E^{\ast }$ are weak$^{\ast }$
sequentially continuous, then every almost limited solid set in $E$ is
limited.

\item If $E$ has \textit{the property (d) and }every almost limited solid
set in $E$ is limited, then the lattice operations of $E^{\ast }$ are weak$%
^{\ast }$ sequentially continuous.
\end{enumerate}
\end{theorem}

\begin{proof}
$(1)$ Let $A\subset E$ be an almost limited solid set. Let $\left(
f_{n}\right) \subset E^{\ast }$ be an arbitrary sequence such that $f_{n}%
\overset{w^{\ast }}{\rightarrow }0$. By hypothesis, $\left\vert
f_{n}\right\vert \overset{w^{\ast }}{\rightarrow }0$. Let $\left(
x_{n}\right) \subset A^{+}$ be a disjoint sequence. By Lemma 2.2 of \cite%
{AB2} there exists a disjoint sequence $\left( g_{n}\right) \subset E^{\ast
} $ such that $\left\vert g_{n}\right\vert \leq \left\vert f_{n}\right\vert $
and $g_{n}\left( x_{n}\right) =f_{n}\left( x_{n}\right) $ for every $n$. It
is easy to see that $g_{n}\overset{w^{\ast }}{\rightarrow }0$. As $A$ is
almost limited $\sup_{x\in A}\left\vert g_{n}(x)\right\vert \rightarrow 0$.
From the inequality $\left\vert f_{n}\left( x_{n}\right) \right\vert
=\left\vert g_{n}\left( x_{n}\right) \right\vert \leq \sup_{x\in
A}\left\vert g_{n}(x)\right\vert $, we conclude that $f_{n}\left(
x_{n}\right) \rightarrow 0$.

Now by Lemma \ref{DF}, we have $\sup_{x\in A}\left\vert f_{n}\right\vert
\left( x\right) =\sup_{x\in A}\left\vert f_{n}\right\vert \left( \left\vert
x\right\vert \right) \rightarrow 0$. Thus by the inequality $\left\vert
f_{n}\left( x\right) \right\vert \leq \left\vert f_{n}\right\vert \left(
\left\vert x\right\vert \right) $ we see that $\sup_{x\in A}\left\vert
f_{n}\left( x\right) \right\vert \rightarrow 0$. Then $A$ is limited.

$\left( 2\right) $ Since $E$ has the property (d) then each order interval $%
\left[ -x,x\right] $\ of $E$ is almost limited\ (Proposition \ref{order
interval} (1)), and by our hypothesis $\left[ -x,x\right] $\ is limited. So,
the lattice operations of $E^{\ast }$ are weak$^{\ast }$ sequentially
continuous (Proposition \ref{order interval} (2)).
\end{proof}

The next main result, gives some equivalent conditions for $T(A)$ to be
almost limited set where $A$ is a norm bounded solid subset of $E$ and $%
T:E\rightarrow F$ is an order bounded operator.

\begin{theorem}
\label{T(A)almostlimited}Let $E$ and $F$ be two Banach lattices such that
the lattice operations of $E^{\ast }$ are sequentially weak$^{\ast }$
continuous or $F$ has the property $(\text{d})$. Then for an order bounded
operator $T:E\rightarrow F$ and a norm bounded solid subset $A\subset E$,
the following assertions are equivalent:

\begin{enumerate}
\item $T(A)$ is almost limited.

\item $f_{n}\left( T\left( x_{n}\right) \right) \rightarrow 0$ for every
disjoint sequence $\left( x_{n}\right) \subset A^{+}$ and every weak$^{\ast
} $ null disjoint sequence $\left( f_{n}\right) \subset F^{\ast }$.
\end{enumerate}

If $F$ has the property $(\text{d})$, we may add:

\begin{enumerate}
\item[$\left( 3\right) $] $f_{n}\left( T\left( x_{n}\right) \right)
\rightarrow 0$ for every disjoint sequence $\left( x_{n}\right) \subset
A^{+} $ and every weak$^{\ast }$ null disjoint sequence $\left( f_{n}\right)
\subset \left( F^{\ast }\right) ^{+}$.
\end{enumerate}
\end{theorem}

\begin{proof}
$\left( 1\right) \Rightarrow \left( 2\right) $ Follows from the inequality $%
\left\vert f_{n}(T\left( x_{n}\right) )\right\vert \leq \sup_{y\in T\left(
A\right) }\left\vert f_{n}(y)\right\vert $.

$\left( 2\right) \Rightarrow \left( 1\right) $ Let $\left( f_{n}\right)
\subset F^{\ast }$ be a disjoint sequence such that $f_{n}\overset{w^{\ast }}%
{\rightarrow }0$.

We claim that $\left\vert T^{\ast }(f_{n})\right\vert \overset{w^{\ast }}{%
\rightarrow }0$ holds in $E^{\ast }$. In fact,

- if the lattice operations of $E^{\ast }$ are sequentially weak$^{\ast }$
continuous then $\left\vert T^{\ast }(f_{n})\right\vert \overset{w^{\ast }}{%
\rightarrow }0$ (since $T^{\ast }(f_{n})\overset{w^{\ast }}{\rightarrow }0$
holds in $E^{\ast }$).

- if $F$ has the property $(\text{d})$, then $\left\vert f_{n}\right\vert 
\overset{w^{\ast }}{\rightarrow }0$. Let $x\in E^{+}$ and pick some $y\in
F^{+}$ such that $T\left[ -x,x\right] \subseteq \left[ -y,y\right] $ (i.e., $%
\left\vert T\left( u\right) \right\vert \leq y$ for all $u\in \left[ -x,x%
\right] $). Thus 
\begin{eqnarray*}
0 &\leq &\left\vert T^{\ast }\left( f_{n}\right) \right\vert \left( x\right)
\\
&=&\sup \left\{ \left\vert \left( T^{\ast }f_{n}\right) \left( u\right)
\right\vert :\left\vert u\right\vert \leq x\right\} \\
&=&\sup \left\{ \left\vert f_{n}\left( T\left( u\right) \right) \right\vert
:\left\vert u\right\vert \leq x\right\} \\
&\leq &\left\vert f_{n}\right\vert \left( y\right)
\end{eqnarray*}%
As $\left\vert f_{n}\right\vert \overset{w^{\ast }}{\rightarrow }0$, we have 
$\left\vert f_{n}\right\vert \left( y\right) \rightarrow 0$ and hence $%
\left\vert T^{\ast }f_{n}\right\vert \left( x\right) \rightarrow 0$. So $%
\left\vert T^{\ast }f_{n}\right\vert \overset{w^{\ast }}{\rightarrow }0$.

On the other hand, by $\left( 2\right) $, $\left( T^{\ast }\left(
f_{n}\right) \right) \left( x_{n}\right) =f_{n}\left( T\left( x_{n}\right)
\right) \rightarrow 0$ for every disjoint sequence $\left( x_{n}\right)
\subset A^{+}$. Now by Lemma \ref{DF}, $\sup_{x\in A}\left\vert T^{\ast
}(f_{n})\right\vert (x)\rightarrow 0$ and hence%
\begin{equation*}
\sup_{y\in T(A)}\left\vert f_{n}(y)\right\vert =\sup_{x\in A}\left\vert
T^{\ast }(f_{n})(x)\right\vert \rightarrow 0.
\end{equation*}%
Then $T\left( A\right) $ is almost limited.

$\left( 2\right) \Rightarrow \left( 3\right) $ Obvious.

$\left( 3\right) \Rightarrow \left( 2\right) $ If $\left( f_{n}\right) $ is
a disjoint weak$^{\ast }$ null sequence in $E^{\ast }$ then $\left\vert
f_{n}\right\vert \overset{w^{\ast }}{\rightarrow }0$ and hence from the
inequalities $f_{n}^{\,+}\leq \left\vert f_{n}\right\vert $ and $%
f_{n}^{\,-}\leq \left\vert f_{n}\right\vert $, the sequences $(f_{n}^{\,+})$%
, $(f_{n}^{\,-})$ are weak$^{\ast }$ null. Finally, by $\left( 3\right) $, $%
\lim f_{n}\left( T\left( x_{n}\right) \right) =\lim \left[ f_{n}^{\,+}\left(
T\left( x_{n}\right) \right) -f_{n}^{\,-}\left( T\left( x_{n}\right) \right) %
\right] =0$ for every disjoint sequence $\left( x_{n}\right) \subset A^{+}$.
\end{proof}

Note that a norm bounded subset $A$ of a Banach lattice $E$ is almost
limited if and only if $f_{n}\left( x_{n}\right) \rightarrow 0$ for every
sequence $\left( x_{n}\right) \subset A$ and for every weak$^{\ast }$ null
disjoint sequence $\left( f_{n}\right) \subset E^{\ast }$.

However, if we take $E=F$ and $T$ the identity operator on $E$ in Theorem %
\ref{T(A)almostlimited}, we obtain the following characterization of solid
almost limited sets which is a generalization of Theorem 2.5 of \cite{chen}.

\begin{corollary}
\label{corT(A)almostlimited}Let $E$ be a Banach lattice satisfying the
property $(\text{d})$. Then for a norm bounded solid subset $A\subset E$,
the following assertions are equivalent:

\begin{enumerate}
\item $A$ is almost limited.

\item $f_{n}\left( x_{n}\right) \rightarrow 0$ for every disjoint sequence $%
\left( x_{n}\right) \subset A^{+}$ and every weak$^{\ast }$ null disjoint
sequence $\left( f_{n}\right) \subset E^{\ast }$.

\item $f_{n}\left( x_{n}\right) \rightarrow 0$ for every disjoint sequence $%
\left( x_{n}\right) \subset A^{+}$ and every weak$^{\ast }$ null disjoint
sequence $\left( f_{n}\right) \subset \left( E^{\ast }\right) ^{+}$.
\end{enumerate}
\end{corollary}

If the lattice operations of $E^{\ast }$ are sequentially weak$^{\ast }$
continuous, we obtain the following characterization of solid limited sets.

\begin{corollary}
Let $E$ be a Banach lattice. If the lattice operations of $E^{\ast }$ are
sequentially weak$^{\ast }$ continuous, then for a norm bounded solid subset 
$A\subset E$, the following assertions are equivalent:

\begin{enumerate}
\item $A$ is limited.

\item $f_{n}\left( x_{n}\right) \rightarrow 0$ for every disjoint sequence $%
\left( x_{n}\right) \subset A^{+}$ and every weak$^{\ast }$ null disjoint
sequence $\left( f_{n}\right) \subset E^{\ast }$.

\item $f_{n}\left( x_{n}\right) \rightarrow 0$ for every disjoint sequence $%
\left( x_{n}\right) \subset A^{+}$ and every weak$^{\ast }$ null disjoint
sequence $\left( f_{n}\right) \subset \left( E^{\ast }\right) ^{+}$.
\end{enumerate}
\end{corollary}

\begin{proof}
Is a simple consequence of Corollary \ref{corT(A)almostlimited} combined
with Theorem \ref{solidalmostlimited}.
\end{proof}

\section{Weak Dunford-Pettis$^{\ast }$ property and dual Schur property}

Recall that a Banach lattice $E$ is called to have the \emph{weak
Dunford-Pettis}$^{\ast }$\emph{\ property} (\emph{wDP}$^{\ast }$\emph{\
property} for short) if every relatively weakly compact set in $E$ is almost
limited \cite[Definition 3.1]{chen}.

In other words, $E$ has the wDP$^{\ast }$ property if and only if for each
weakly null sequence $(x_{n})$ in $E$ and each disjoint weak$^{\ast }$ null
sequence $\left( f_{n}\right) $ in $E^{^{\ast }}$, $f_{n}(x_{n})\rightarrow
0 $.

Also, we say that a Banach lattice $E$ satisfy the \emph{bi-sequence property%
} if for each disjoint weakly null sequence $(x_{n})\subset E^{+}$ and each
weak$^{\ast }$ null sequence $(f_{n})\subset (E^{\ast })^{+}$, we have $%
f_{n}(x_{n})\rightarrow 0$, equivalently, if for each disjoint weakly null
sequence $(x_{n})\subset E^{+}$ and each disjoint weak$^{\ast }$ null
sequence $(f_{n})\subset (E^{\ast })^{+}$, we have $f_{n}(x_{n})\rightarrow
0 $ \cite[Proposition 2.4]{WN2012}.

Clearly, every Banach lattice with the wDP$^{\ast }$ property has the
bi-sequence property. The converse is false in general. In fact, the Banach
lattice $c$ has the bi-sequence property (each positive weak$^{\ast }$ null
sequence $(f_{n})$ in $c^{\ast }$ is norm null) but does not have the wDP$%
^{\ast }$ property. Indeed, let $f_{n}\in c^{\ast }=\ell ^{1}$\ be defined
as follows $f_{n}=(0,\cdot \cdot \cdot ,0,1_{(2n)},-1_{(2n+1)},0,\cdot \cdot
\cdot )$. Then $(f_{n})$ is a disjoint weak$^{\ast }$null sequence in $%
c^{\ast }$\cite[Example 2.1(2)]{chen}, and clearly, the sequence $\left(
x_{n}\right) $ defined by $x_{n}=(0,\cdot \cdot \cdot ,0,1_{(2n)},0,\cdot
\cdot \cdot )\in c$ is weakly null, but $f_{n}\left( x_{n}\right) =1$ for
all $n$.

For the Banach lattices satisfying the property $(\text{d})$, the concepts
of bi-sequence property and wDP$^{\ast }$ property coincide. The details
follow.

\begin{theorem}
\label{bisequence}For a Banach lattice $E$ satisfying the property $(\text{d}%
)$, the following statements are equivalent:

\begin{enumerate}
\item $E$ has the wDP$^{\ast }$ property.

\item For each disjoint weakly null sequence $(x_{n})\subset E$ and each
disjoint weak$^{\ast }$ null sequence $(f_{n})\subset E^{\ast }$, we have $%
f_{n}(x_{n})\rightarrow 0$.

\item For each disjoint weakly null sequence $(x_{n})\subset E^{+}$ and each
disjoint weak$^{\ast }$ null sequence $(f_{n})\subset (E^{\ast })^{+}$, we
have $f_{n}(x_{n})\rightarrow 0$.

\item The solid hull of every relatively weakly compact set in $E$ is almost
limited.

\item For each weakly null sequence $(x_{n})\subset E^{+}$ and each weak$%
^{\ast }$ null sequence $(f_{n})\subset (E^{\ast })^{+}$, we have $%
f_{n}(x_{n})\rightarrow 0$.

\item $E$\ has the bi-sequence property.
\end{enumerate}
\end{theorem}

\begin{proof}
$(1)\Rightarrow (2)\Rightarrow (3)$ and $(4)\Rightarrow (1)$\ are obvious.

$(3)\Leftrightarrow (5)\Leftrightarrow (6)$ by Proposition 2.4 of \cite%
{WN2012}.

$(3)\Rightarrow (4)$ Using Corollary \ref{corT(A)almostlimited}, it is
enough to repeat the proof of the implication $(3)\Rightarrow (4)$ in \cite[%
Theorem 3.2]{chen}.
\end{proof}

Note that the equivalences $(1)\Leftrightarrow (2)\Leftrightarrow
(3)\Leftrightarrow (4)$ are proved in \cite[Theorem 3.2]{chen} under the
hypothesis that $E$ is $\sigma $-Dedekind complete.

Recall that a Banach lattice $E$ has the \emph{dual positive Schur property }%
($E\in $(\textrm{DPSP})), if every weak$^{\ast }$ null sequence $\left(
f_{n}\right) \subset \left( E^{\ast }\right) ^{+}$ is norm null. This
property was introduced in \cite{aqz} and further developed in \cite{WN2012}%
. We remark here that $E\in $(\textrm{DPSP}) if every disjoint weak$^{\ast }$
null sequence $\left( f_{n}\right) \subset \left( E^{\ast }\right) ^{+}$ is
norm null \cite[Proposition 2.3]{WN2012}. So, it is natural to study the
Banach lattices $E$ satisfying the following property:%
\begin{equation*}
f_{m}\perp f_{k}\text{ in }E^{\ast }\text{ and }f_{n}\overset{w^{\ast }}{%
\rightarrow }0\text{ implies }\left\Vert f_{n}\right\Vert \rightarrow 0.
\end{equation*}

\begin{definition}
A Banach lattice $E$ is called to have the \emph{dual Schur property} ($E\in 
$(DSP)), if each disjoint weak$^{\ast }$ null sequence $\left( f_{n}\right)
\subset E^{\ast }$ is norm null.
\end{definition}

In other words, $E$ has the dual Schur property if and only if the closed
unit ball $B_{E}$ is almost limited, equivalently, each norm bounded set in $%
E$ is almost limited.

Clearly, if $E\in $(\textrm{DSP}) then $E\in $(\textrm{DPSP}) and $E$ has
the wDP$^{\ast }$ property, but the converse is false in general. In fact,
the Banach lattice $c$ has the dual positive Schur property ($c\in $(\textrm{%
DPSP})) but by Example 2.1(2) of \cite[(2)]{chen}, $c\notin $(\textrm{DSP}).
On the other hand, the Banach lattice $\ell ^{1}$ has the wDP$^{\ast }$
property but $\ell ^{1}\notin $(\textrm{DPSP}) (see Proposition 2.1 of \cite%
{WN2012}) and hence $\ell ^{1}\notin $(\textrm{DSP}).

For the Banach lattices satisfying the property $(\text{d})$ (in particular,
if $E$ is $\sigma $-Dedekind complete), the notions of dual Schur property
and dual positive Schur property coincide, and some new characterizations of
the dual positive Schur property can be obtained. The details follow.

\begin{theorem}
\label{DSP}For a Banach lattice $E$ satisfying the property $(\text{d})$,
the following statements are equivalent:

\begin{enumerate}
\item $E\in $(\textrm{DSP})

\item $E\in $(\textrm{DPSP})

\item $E^{\ast }$ has order continuous norm and $E$ has the bi-sequence
property.

\item $E^{\ast }$ has order continuous norm and $E$ has the wDP$^{\ast }$
property.

\item $f_{n}\left( x_{n}\right) \rightarrow 0$ for every bounded and
disjoint sequence $\left( x_{n}\right) \subset E^{+}$ and every disjoint weak%
$^{\ast }$null sequence $\left( f_{n}\right) \subset E^{\ast }$.

\item $f_{n}\left( x_{n}\right) \rightarrow 0$\ for every bounded and
disjoint sequence $\left( x_{n}\right) \subset E^{+}$\ and every disjoint
weak$^{\ast }$null sequence $\left( f_{n}\right) \subset \left( E^{\ast
}\right) ^{+}$.
\end{enumerate}
\end{theorem}

\begin{proof}
$\left( 1\right) \Leftrightarrow \left( 2\right) $ Obvious, as $E$ has the
property $($d$)$.

$\left( 2\right) \Leftrightarrow \left( 3\right) $ Proposition 2.5 \cite%
{WN2012}.

$\left( 3\right) \Leftrightarrow \left( 4\right) $ Follows from Theorem \ref%
{bisequence}.

$\left( 1\right) \Leftrightarrow \left( 5\right) \Leftrightarrow \left(
6\right) $ Follows from Corollary \ref{corT(A)almostlimited} by noting that $%
E\in $(\textrm{DSP}) if and only if the closed unit ball $B_{E}$\ is almost
limited.
\end{proof}

\begin{remark}
If $E\in $(\textrm{DPSP}) then the lattice operation of $E^{\ast }$ are not
weak$^{\ast }$ sequentially continuous when $E$ is infinite dimensional.
This is a simple consequence of Josefson--Nissenzweig theorem.

On the other hand, only the finite dimensional Banach spaces $X$ satisfy the
following property%
\begin{equation*}
\text{ }f_{n}\overset{w^{\ast }}{\rightarrow }0\text{ in }X^{\ast }\text{
implies }\left\Vert f_{n}\right\Vert \rightarrow 0.
\end{equation*}
\end{remark}

\section{Applications: Almost limited operators}

Using the limited sets, Bourgain and Diestel \cite{BD} introduced the class
of limited operators. An operator $T:X\rightarrow Y$ is said to be \emph{%
limited} if $T\left( B_{X}\right) $ is a limited subset of $Y$.
Alternatively, an operator $T:X\rightarrow Y$ is limited if and only if $%
\left\Vert T^{\ast }\left( f_{n}\right) \right\Vert \rightarrow 0$ for every
weak$^{\ast }$ null sequence $\left( f_{n}\right) \subset Y^{\ast }$.
Therefore, operators $T:X\rightarrow E$ that carry closed unit ball to
almost limited sets arise naturally.

\begin{definition}
\label{def1}An operator $T:X\rightarrow E$\ from a Banach space into a
Banach lattice is said to be almost limited, if $T\left( B_{X}\right) $\ is
an almost limited subset of $E$.
\end{definition}

In other words, $T$\ is almost limited if and only if $\left\Vert T^{\ast
}\left( f_{n}\right) \right\Vert \rightarrow 0$\ for every disjoint weak$%
^{\ast }$\ null sequence $\left( f_{n}\right) \subset E^{\ast }$.

Clearly every limited operator $T:X\rightarrow E$ is almost limited. But the
converse is not true in general. In fact, since $B_{l^{\infty }}$ is almost
limited, the identity operator $I:l^{\infty }\rightarrow l^{\infty }$ is
almost limited. But $I$ is not limited as $B_{l^{\infty }}$ is not a limited
set.

Recall that an operator $T:X\rightarrow E$ is called \emph{L-weakly compact}
if $T\left( B_{X}\right) $ is L-weakly compact. Note that, the identity
operator $I:l^{\infty }\rightarrow l^{\infty }$ is almost limited but fail
to be L-weakly compact (resp. compact).

Also, an operator $T:X\rightarrow E$ is said to be semi-compact if $T(B_{X})$
is almost order bounded, that is, for each $\varepsilon >0$ there exists
some $u\in E^{+}$ such that $T(B_{X})\subset \lbrack -u,u]+\varepsilon B_{E}$%
.

The next theorem deals the relationship between almost limited and L-weakly
compact (resp.semi-compact) operators.

\begin{theorem}
\label{almostlimitedisLW}Let $X$ be a non-zero Banach space and let $E$ be a
Banach lattice. Then the following statements hold.

\begin{enumerate}
\item Each L-weakly compact operator $T:X\rightarrow E$ is almost limited.

\item Each almost limited operator $T:X\rightarrow E$ is L-weakly compact if
and only if the norm of $E$ is order continuous.

\item If $E$ is discrete with order continuous norm, then each almost
limited operator $T:X\rightarrow E$ is compact.

\item If $E$\ satisfy the property $($\textrm{d}$)$\ then each semi-compact
operator $T:X\rightarrow E$\ is almost limited.

\item If the lattice operations of $E^{\ast }$\ are weak$^{\ast }$\
sequentially continuous then each semi-compact operator $T:X\rightarrow E$\
is limited.
\end{enumerate}
\end{theorem}

\begin{proof}
$\left( 1\right) $ follows immediately from Theorem 2.6(1) of \cite{chen}.

$\left( 2\right) $ The \textquotedblleft if\textquotedblright\ part follows
immediately from Theorem 2.6(2) of \cite{chen}. For the \textquotedblleft
only if\textquotedblright\ part assume by way of contradiction that the norm
on $E$ is not order continuous. Then, there exists some $y\in E^{+}$ and a
disjoint sequence $(y_{n})\subset \lbrack 0,y]$ which does not converge to
zero in norm. For non-zero Banach space $X$, choosing $u\in X$\ such that $%
\left\Vert u\right\Vert =1$\ and let $\varphi \in X^{\ast }$\ such that $%
\left\Vert \varphi \right\Vert =1$\ and $\varphi \left( u\right) =\left\Vert
u\right\Vert $. Now, define the operator $T:X\rightarrow E$\ as follows:%
\begin{equation*}
T\left( x\right) =\varphi \left( x\right) y\qquad \text{for every }x\in E.
\end{equation*}%
Clearly, $T$ is almost limited as it is compact (its rank is one). Hence, By
hypothesis $T$ is L-weakly compact. Note that $(y_{n})$ is a disjoint
sequence in the solid hull of $T(B_{E})$. But the L-weak compactness of $T$
imply that $\left\Vert y_{n}\right\Vert \rightarrow 0$, which is a
contradiction.

$\left( 3\right) $ Assume that $T:X\rightarrow E$ is almost limited. Since
the norm of $E$ is order continuous, then by (2) $T$ is L-weakly compact,
and by Theorem 5.71 of \cite{AB3} $T$ is semi-compact. So, for each $%
\varepsilon >0$ there exists some $u\in E^{+}$ such that $T(B_{X})\subset
\lbrack -u,u]+\varepsilon B_{E}$. Since $E$ is discrete with order
continuous norm, it follows from Theorem 6.1 of \cite{WN4} that $[-u,u]$ is
compact. Therefore, $T$ is compact.

$\left( 4\right) $ and $\left( 5\right) $ follows immediately from
Proposition \ref{almostorderbounded}.
\end{proof}

If the Banach lattice $E$ is $\sigma $-Dedekind complete, we obtain the
following result.

\begin{theorem}
\label{compact}For a $\sigma $-Dedekind complete Banach lattice $E$ the
following assertions are equivalent:

\begin{enumerate}
\item $E$ is discrete with order continuous norm.

\item Each almost limited operator from an arbitrary Banach space $X $ into $%
E$ is limited.

\item Each almost limited operator from $\ell ^{1}$ into $E$ is limited.
\end{enumerate}
\end{theorem}

\begin{proof}
$\left( 1\right) \Rightarrow \left( 2\right) $ Follows from Theorem \ref%
{almostlimitedisLW} (3) because every relatively compact subset is limited.

$\left( 2\right) \Rightarrow \left( 3\right) $ Obvious.

$\left( 3\right) \Rightarrow \left( 1\right) $ According to Theorem 6.6 of 
\cite{WN4} it suffices to show that the lattice operations of $E^{\prime }$\
are weak* sequentially continuous.

Indeed, assume by way of contradiction that the lattice operations of $%
E^{\ast }$ are not weak* sequentially continuous, and let $(f_{n})$ be a
sequence in $E^{\ast }$ with $f_{n}\overset{w^{\ast }}{\rightarrow }0$ but $%
\left\vert f_{n}\right\vert \overset{w^{\ast }}{\nrightarrow }0$. So there
is some $x\in E^{+}$ with $\left\vert f_{n}\right\vert (x)\nrightarrow 0$.
By choosing a subsequence we may suppose that there is $\varepsilon >0$ with 
$\left\vert f_{n}\right\vert (x)>\varepsilon $ for all $n\in \mathbb{N}$. In
view of $\left\vert f_{n}\right\vert (x)=\sup \{f_{n}(u):\left\vert
u\right\vert \leq x\}$, for every $n$ there exist some $\left\vert
x_{n}\right\vert \leq x$ such that $2f_{n}(x_{n})\geq \left\vert
f_{n}\right\vert (x)>\varepsilon $. Now, consider the operator $T:\ell
^{1}\rightarrow E$\ defined by%
\begin{equation*}
T\left( (\lambda _{n})\right) =\sum_{n=1}^{\infty }\lambda _{n}x_{n}.
\end{equation*}%
and note that its adjoint is given by $T^{\ast }\left( f\right)
=(f(x_{n}))_{n=1}^{\infty }$ for each $f\in E^{\ast }$. We claim that $T$ is
almost limited. In fact, from%
\begin{equation*}
\left\vert \sum_{n=1}^{m}\lambda _{n}x_{n}\right\vert \leq
\sum_{n=1}^{m}\left\vert \lambda _{n}\right\vert \cdot \left\vert
x_{n}\right\vert \leq \left( \sum_{n=1}^{m}\left\vert \lambda
_{n}\right\vert \right) x\leq x
\end{equation*}%
for every $(\lambda _{n})\in B_{\ell ^{1}}$ and every $m\in \mathbb{N}$, we
see that $T\left( B_{\ell ^{1}}\right) \subseteq \left[ -x,x\right] $. Since 
$E$ is $\sigma $-Dedekind complete then $\left[ -x,x\right] $ (and hence $%
T\left( B_{\ell ^{1}}\right) $ itself) is almost limited. Therefore, by our
hypothesis the operator $T$ limited, and hence $\left\Vert T^{\ast
}(f_{n})\right\Vert _{\infty }\rightarrow 0$, which contradicts $\left\Vert
T^{\ast }(f_{n})\right\Vert _{\infty }\geq \left\vert
f_{n}(x_{n})\right\vert >\frac{\varepsilon }{2}$. This completes the proof
of the theorem.
\end{proof}

The next characterization of the order bounded almost limited operators
between two Banach lattices follows immediately from Theorem \ref%
{T(A)almostlimited}.

\begin{proposition}
\label{almostlimitedop}Let $E$ and $F$ be two Banach lattices such that the
lattice operations of $E^{\ast }$ are sequentially weak$^{\ast }$ continuous
or $F$ has the property $($\textrm{d}$)$. Then for an order bounded operator 
$T:E\rightarrow F$ the following assertions are equivalent:

\begin{enumerate}
\item $T$ is almost limited.

\item $f_{n}\left( T\left( x_{n}\right) \right) \rightarrow 0$ for every
norm bounded disjoint sequence $\left( x_{n}\right) \subset E^{+}$ and every
weak$^{\ast }$ null disjoint sequence $\left( f_{n}\right) \subset F^{\ast }$%
.
\end{enumerate}

If $F$ has the property $($\textrm{d}$)$, we may add:

\begin{enumerate}
\item[$\left( 3\right) $] $f_{n}\left( T\left( x_{n}\right) \right)
\rightarrow 0$ for every norm bounded disjoint sequence $\left( x_{n}\right)
\subset E^{+}$ and every weak$^{\ast }$ null disjoint sequence $\left(
f_{n}\right) \subset \left( F^{\ast }\right) ^{+}$.
\end{enumerate}
\end{proposition}

Let us recall that an operator $T:E\rightarrow X$ is said to be \emph{%
M-weakly compact} if $\left\Vert T\left( x_{n}\right) \right\Vert
\rightarrow 0$ holds for every norm bounded disjoint sequence $\left(
x_{n}\right) $ of $E$.

Note that a M-weakly compact operator need not be almost limited. In fact,
we know that every operator from $\ell ^{\infty }$ into $c_{0}$ is M-weakly
compact and by a result in \cite{Wn}, there exists a non regular operator $%
T:\ell ^{\infty }\rightarrow c_{0}$, which is certainly not compact. Thus by
Theorem \ref{almostlimitedisLW} (3), $T$ is not almost limited.

However, by Proposition \ref{almostlimitedop} we have the following easy
result and omit the proof.

\begin{proposition}
Let $E$ and $F$ be two Banach lattices. If the lattice operations of $%
E^{\ast }$ are sequentially weak$^{\ast }$ continuous or $F$ has the
property $($\textrm{d}$)$, then each order bounded M-weakly-compact operator 
$T:E\rightarrow F$ is almost limited.
\end{proposition}

\end{document}